\documentclass[12pt]{amsart}
\usepackage{latexsym}
\usepackage{amssymb}

\usepackage{amsthm}
\usepackage{latexsym}
\usepackage{longtable}
\usepackage{epsfig}
\usepackage{amsmath}
\usepackage{hhline}
\usepackage{color}

\usepackage{listings}
\lstdefinelanguage{GAP}{%
  morekeywords={%
    Assert,Info,IsBound,QUIT,%
    TryNextMethod,Unbind,and,break,%
    continue,do,elif,%
    else,end,false,fi,for,%
    function,if,in,local,%
    mod,not,od,or,%
    quit,rec,repeat,return,%
    then,true,until,while%
  },%
  sensitive,%
  morecomment=[l]\#,%
  morestring=[b]",%
  morestring=[b]',%
}[keywords,comments,strings]

\usepackage[T1]{fontenc}
\usepackage[variablett]{lmodern}
\usepackage{xcolor}
\newtheorem{theorem}{Theorem}[section]
\newtheorem{lemma}[theorem]{Lemma}
\newtheorem{proposition}[theorem]{Proposition}

\theoremstyle{definition}

\def\fis#1{\mathord{\mathrm{Fi}_{#1}}}

\def\ja#1{\mathord{\mathrm{J}_{#1}}}

\def\ru{\mathord{\mathrm{Ru}}}

\def\ly{\mathord{\mathrm{Ly}}}

\def\ON{\mathord{\mathrm{O}^{^{\textstyle\mathrm{,}}}\mathrm{N}}}

\def\tF#1{\mathord{{}^2\mathrm{F}_4(#1)}}
\def\tD#1{\mathord{{}^3\mathrm{D}_4(#1)}}

\def\tE#1{\mathord{{}^2\mathrm{E}_6(#1)}}
\def\tB#1{\mathord{{}^2\mathrm{B}_2(#1)}}
\def\tG#1{\mathord{{}^2\mathrm{G}_2(#1)}}

\theoremstyle{remark}

\begin{document}
\title[Simple groups with short Galois orbits]{Finite simple groups with short Galois orbits on conjugacy classes}

\author[Bovdi]{Victor Bovdi}
\address{Department of Mathematical Sciences, UAEU, Al-Ain, United Arab Emirates}
\email{vbovdi@gmail.com}

\author[Breuer]{Thomas Breuer}
\address{Lehrstuhl D f\"ur Mathematik, RWTH Aachen University, 52065 Aachen, Germany}
\email{sam@math.rwth-aachen.de}

\author[Mar\'oti]{Attila Mar\'oti}
\address{Alfr\'ed R\'enyi Institute of Mathematics, Hungarian Academy of Sciences, Re\'alta-noda utca 13-15, H-1053, Budapest, Hungary}
\email{maroti.attila@renyi.mta.hu}


\thanks{The work of the first author was supported by UAEU UPAR grant G00002160. The second author gratefully acknowledges support by the German Research Foundation (DFG) within the SFB-TRR 195 ``Symbolic Tools in Mathematics and their Application''. The work of the third author on the project leading to this application has received funding from the European Research Council (ERC) under the European Union's Horizon 2020 research and innovation programme (grant agreement No. 741420), was supported by the J\'anos Bolyai Research Scholarship of the Hungarian Academy of Sciences and also by the National Research, Development and Innovation Office (NKFIH) Grant No.~K115799}

\begin{abstract}
All finite simple groups are determined with the property that every Galois orbit on conjugacy classes has size at most $4$. From this we list all finite simple groups $G$ for which the normalized group of central units of the integral group ring $\mathbb{Z}G$ is an infinite cyclic group.
\end{abstract}
\maketitle

\begin{center}
{\it Dedicated to Professors \'Agnes Szendrei and M\'aria B. Szendrei \\ on the occasion of their birthdays.}
\end{center}

\bigskip

\section{Introduction}

According to \cite[Satz 14.9, p.\,545]{Huppert_book} (see also \cite[Theorem 3.2, p. \,32]{Bovdi_book}), the group of units of the ring of integers of the center of the rational group algebra $\mathbb{Q}G$ of a finite group $G$ is isomorphic to the direct product of the groups of units of the rings of integers of the fields $\mathbb{Q}(\chi)$ where $\chi$ runs over a set of representatives of the orbits of algebraically conjugate irreducible complex characters of $G$.

Moreover, for a finite group $G$, the group of central units of the integral group ring  $\mathbb{Z}G$ is isomorphic, by the theorem of Berman and Higman (see \cite[1.1, p.\,5]{Artamonov_Bovdi} and \cite[Theorem 2.2, p.\,23]{Bovdi_book}), to $\langle - 1 \rangle \times Z(G) \times R(G)$ where $R(G) = 1$ or $R(G)$ is a finitely generated torsion-free abelian subgroup of $\mathbb{Z}G$ with elements of augmentation $1$. A systematic study of central units was first launched by A.~Bovdi (see for example \cite[Chapter 8]{Bovdi_book}). In particular, in \cite[Lemma 8.1, p.\;81]{Bovdi_book} he describes the basic situation when $R(G)=1$. For a finite group $G$ we have $R(G) = 1$ if and only if the character field $\mathbb{Q}(\chi)$ of each complex irreducible character $\chi$ of $G$ is either $\mathbb{Q}$ or imaginary quadratic (of the form $\mathbb{Q}(\sqrt{d})$ for some $d < 0$ in $\mathbb{R}$). This happens if and only if every generator of every cyclic group $\langle g\mid g \in G\rangle$ is conjugate in $G$ to $g$ or to $g^{-1}$. We give a generalization of this statement in Proposition \ref{bounded}.

As a natural consequence of A.~Bovdi's result, a not necessarily finite group $G$ is called a {\it cut}-group if all central units of $\mathbb{Z}G$ are trivial (of the form $\pm g$ where $g \in Z(G)$), see \cite{Bakshi_Maheshwary_Passi}. The first results on {\it cut}-groups were obtained by A.~Bovdi and Patay. Recent results and references on this topic can be found in \cite{Bachle, Bakshi_Maheshwary_Passi} and \cite{Bakshi_Maheshwary_Passi_II}.

The description of finite nilpotent {\it cut}-groups of class 2 was obtained by Patay (see  \cite[Theorem 8.2, p.\,83]{Bovdi_book}, \cite{Patay_I} and \cite{Patay}). B\"achle \cite[Theorem 1.2]{Bachle} proved that any prime divisor of the order of a finite solvable {\it cut}-group is $2$, $3$, $5$, or $7$.


In the 1970's, Patay (as a PhD student of A.~Bovdi) was the first to study the question of when a finite simple group is a {\it cut}-group (he answered this for alternating groups). Afterwards, several people continued to investigate this question and for a long time the most extensive results on this topic were obtained by Aleev and his students (see \cite{Aleev_4, Aleev_dissertation, Aleev_Ishechkina, Aleev_Kargapolov_Sokolov, Aleev_Kargapolov_Sokolov, Molodorich}). The (finite) list of finite simple {\it cut}-groups was deduced in \cite[Theorem 5.1]{BCJM} from a longer list of groups obtained in \cite{AD}. Note that in part (ii) of our Theorem \ref{mainmain} we give a corrected list of groups presented in \cite{AD}.

For a finite group $G$ we denote the rank of $R(G)$ by $\mathfrak{r}_{\mathbb{Z}}(G)$ or simply by $\mathfrak{r}_{\mathbb{Z}}$.

An aim of this paper is to determine all finite simple groups $G$ for which $\mathfrak{r}_{\mathbb{Z}}(G) = 1$.

\begin{theorem}
\label{rank1}
A finite simple group $G$ satisfies $\mathfrak{r}_{\mathbb{Z}}(G) = 1$ if and only if $G = A_n$ with $n \in \{\; 5, 6, 10, 11, 13, 16, 17, 21, 25\; \}$ or $G \in \{\; C_{5}, \mathrm{PSL}(2,11), \mathrm{SL}(3,3), \mathrm{PSL}(3,4), \\ \mathrm{PSL}(4,3), \mathrm{PSp}(6,3), \mathrm{Sp}(8,2), \mathrm{P\Omega}(7,3), \mathrm{P\Omega}^{+}(8,3), \mathrm{G}_{2}(3), \mathrm{F}_{4}(2), \tE{2}, \fis{22}\; \}$.
\end{theorem}

All groups $G$ listed in Theorem \ref{rank1} were shown to satisfy $\mathfrak{r}_{\mathbb{Z}}(G) = 1$ in \cite[p. 328-329]{Aleev_dissertation}. Theorem \ref{rank1} in the special case of alternating groups was obtained in \cite{Aleev_Kargapolov_Sokolov}.

In general, the following asymptotic statement holds.

\begin{theorem}
\label{secondmain}
There exists a universal constant $c>0$ such that whenever $G$ is a finite simple group then $k(G)^{c} < \mathfrak{r}_{\mathbb{Z}}(G) < k(G)$ provided that $\mathfrak{r}_{\mathbb{Z}}(G) > 1$ where $k(G)$ denotes the number of conjugacy classes of $G$. In particular, $\mathfrak{r}_{\mathbb{Z}}(G) \to \infty$ as the orders of the finite simple groups $G$ tend to infinity.
\end{theorem}

We deduce Theorem \ref{rank1} from a more general result, Theorem \ref{mainmain}.

Let $G$ be a finite group and $e$ its exponent. Let $\epsilon$ be a primitive $e$-th root of unity. The Galois group $\mathcal{G} = \mathrm{Gal}(\mathbb{Q}[\epsilon]:\mathbb{Q})$, isomorphic to the unit group of $\mathbb{Z}/e\mathbb{Z}$, acts naturally on the set $\mathrm{Irr}(G)$ of complex irreducible characters of $G$ and on the set $\mathcal{C}$ of conjugacy classes of $G$. Since for every element $\alpha$ in $\mathcal{G}$ the number of fixed points of $\alpha$ on $\mathcal{C}$
and on $\mathrm{Irr}(G)$ coincide, the number of $\mathcal{G}$-orbits on $\mathcal{C}$ and on $\mathrm{Irr}(G)$ are the same. We keep the notation $\mathcal{G}$ and $\mathcal{C}$ throughout this paper.

For a finite group $G$ let $f(G)$ denote the maximal length of an orbit of $\mathcal{G}$ on $\mathcal{C}$. The group $G$ is rational if and only if $f(G) =1$. Another goal of the present paper is to classify finite simple groups $G$ for which $f(G) \leq 4$.

Feit and Seitz \cite{FeitSeitz} proved that the only rational non-abelian finite simple groups are $\mathrm{Sp}(6,2)$ and $\mathrm{SO^{+}}(8,2)$. In \cite{AD} a list of groups for which $f(G) \leq 2$ was given. Since this list contains errors, we reproduce it in (i) and (ii) in the result below. (Note that the groups $\tD{2}$,\;  $\tD{3}$,\;  $\tB{8}$, \; $\tB{32}$, \;  $\tG{27}$, \; $\tF{2}'$ appear incorrectly in \cite[Table 1]{AD} and $\mathrm{G}_{4}(2)$ is not displayed.)

\begin{theorem}\label{mainmain}
Let $G$ be a finite simple group.
\begin{enumerate}
\item[(i)] $f(G) = 1$ if and only if $G \in \{\; C_{2}, \; \mathrm{Sp}(6,2), \; \mathrm{SO^{+}}(8,2)\; \}$.

\item[(ii)] $f(G) = 2$ if and only if $G = A_{n}$ for $n \geq 5$ or $G \in  \{ C_{3}, \; \mathrm{PSL(2,7)}, \;  \mathrm{PSL}(2,11), \\ \mathrm{PSL}(3,4), \; \mathrm{PSp}(6,3), \; \mathrm{Sp}(8,2), \; \mathrm{SU}(3,3),\;  \mathrm{PSU}(3,5), \; \mathrm{SU}(4,2) \cong \mathrm{PSp}(4,3), \\ \mathrm{PSU}(4,3),\;  \mathrm{SU}(5,2),\;  \mathrm{PSU}(6,2), \; \mathrm{P\Omega}(7,3), \; \mathrm{P\Omega}^{+}(8,3), \; \mathrm{F}_{4}(2), \;  \mathrm{G}_{2}(3), \;  \mathrm{G}_{2}(4), \\ \tE{2} \}$ or $G$ is one of twenty sporadic simple groups different from $\ly$ and different from the five sporadic groups listed in (iii).
	
\item[(iii)] $f(G) = 3$ if and only if $G \in \{\; \mathrm{SL}(2,8), \;  \mathrm{PSL}(2,13), \; \mathrm{PSL}(2,17), \; \mathrm{PSL}(2,19), \\ \mathrm{Sp}(10,2), \; \mathrm{PSp}(4,5), \; \tD{2}, \; \tB{8} \}$ or $G \in \{ \ja{1}, \; \ja{3}, \; \ja{4}, \; \ON, \; \ru \;\}$.

\item[(iv)] $f(G) = 4$ if and only if $G \in \{\; C_{5}, \mathrm{PSL}(2,29), \mathrm{PSL}(2,31), \mathrm{SL}(3,3), \mathrm{PSL}(4,3), \\ \mathrm{Sp}(12,2), \mathrm{Sp}(4,4), \mathrm{SU}(3,4), \mathrm{\Omega}^{-}(8,2), \mathrm{\Omega}^{-}(10,2), \mathrm{\Omega}^{+}(12,2), \tF{2}'\; \}$.
\end{enumerate}
\end{theorem}

Our strategy for classifying the finite simple groups $G$ with a given value $f(G)$ is to use upper bounds for certain parameters to get a list of candidates, and to check each of these candidates explicitly. Thus this last step is the problem to determine simple groups $G$ with given value $f(G)$.

In Section 2 we explain the technical background of the paper. Proposition \ref{bounded} is a generalization of an old result of A. Bovdi. These ideas are used to deduce a key tool of the work, Lemma \ref{cyclic}. Section 3 introduces two functions with which computations are performed using known character tables of finite simple groups. These results as well as Proposition \ref{individual} are later used in Sections 5 and 6. After a short section on reductions, classical simple groups are treated in Section 5 and exceptional simple groups of Lie type in Section 6. Section 7 is on alternating groups. In Section 8 we establish Theorem \ref{secondmain}.

\section{Background}

In this section we present some background concerning the invariant $\mathfrak{r}_{\mathbb{Z}}$. In particular, we prove Proposition \ref{bounded} which is a generalization of \cite[Lemma 8.1, p.\;81]{Bovdi_book}.


The following formula was well-known in the 1970's, see Patay \cite{Patay_I}.

\begin{equation}\label{E:1}
\mathfrak{r}_{\mathbb{Z}} = h_{\mathbb{R}} +\textstyle \frac{1}{2} h_{\mathbb{C}} - n_{G}.
\end{equation}
Here $h_{\mathbb{R}}$ denotes the number of real-valued complex irreducible characters of $G$,\quad  $h_{\mathbb{C}} = |\mathrm{Irr}(G)| - h_{\mathbb{R}}$, and $n_{G}$ is the number of $\mathcal{G}$-orbits on $\mathcal{C}$ and on $\mathrm{Irr}(G)$.

For $a \in G$ let $|a|$ denote the order of $a$ and let $a^{G}$ be the conjugacy class of $a$. The $\mathbb{Q}$-class $\{a^G\}_\mathbb{Q}$ of $a$ is defined to be the set \; $\cup_{s=1,\; (s, |G|)=1}^{|a|-1} (a^s)^G$.\; The number $n_{\mathbb{Q}}$ of $\mathbb{Q}$-classes in $G$ is equal to $n_{G}$. The $\mathbb{R}$-class $\{a^G\}_\mathbb{R}$ of $a\in G$ is defined to be $a^G \cup (a^{-1})^G$. Let $n_\mathbb{R}$ be the number of $\mathbb{R}$-classes of $G$.

From these we obtain the following.

\begin{proposition}\label{Prop:2.1}
If $G$ is a finite group, then $\mathfrak{r}_{\mathbb{Z}} = n_{\mathbb{R}} - n_{\mathbb{Q}}$.
\end{proposition}

\begin{proof}
It is sufficient to see by \eqref{E:1} that $\frac{1}{2}\big(|\mathrm{Irr}(G)| + h_{\mathbb{R}}\big) = n_{\mathbb{R}}$. This is clear since $|\mathrm{Irr}(G)| = |\mathcal{C}|$ and $h_{\mathbb{R}}$ is equal to the number of real conjugacy classes in $G$ by \cite[Theorem 11.9, p. 265]{bible}.
\end{proof}

We emphasize that Proposition \ref{Prop:2.1} was well known for decades.

Let $\mathcal{A}$ be the set of conjugacy classes $g^{G}$ of $G$ such that some generator of $\langle g \rangle$ is conjugate in $G$ neither to $g$ nor to $g^{-1}$. The group $\mathcal{G}$ acts on $\mathcal{A}$ in a natural way according to its action on $\mathcal{C}$. Let $a_{2}$ denote the number of orbits in this specified action. Let $\sigma \in \mathcal{G}$ be complex conjugation. It acts as inversion on $\mathcal{C}$ (which action may be trivial). Let $a_{1}$ be the number of orbits of $\langle \sigma \rangle$ on $\mathcal{A}$.
Analogously, let $\mathcal{B}$ be the set of complex irreducible characters $\chi$ of $G$ such that $\mathbb{Q}(\chi)$ is neither rational nor imaginary quadratic. The group $\mathcal{G}$ acts on $\mathcal{B}$ in a natural way. Let the number of orbits of $\mathcal{G}$ in this specified action be $b_{2}$. The element $\sigma$ in $\mathcal{G}$ acts as complex conjugation on $\mathrm{Irr}(G)$. Let the number of orbits of $\langle \sigma \rangle$ on $\mathcal{B}$ be denoted by $b_{1}$.

The following is a generalization of the result of A. Bovdi \cite[Lemma 8.1, p.\;81]{Bovdi_book} already mentioned in the Introduction.

\begin{proposition}
\label{bounded}
Let $G$ be a finite group. Let $a_1$, $a_2$, $b_1$, $b_2$ be as above. Then $2a_{2} \leq a_{1}$ and $2b_{2} \leq b_{1}$. Moreover, $\mathfrak{r}_{\mathbb{Z}}(G) = a_{1}-a_{2} = b_{1}-b_{2}$.


\end{proposition}

\begin{proof}
Let $\mathcal{G}$ denote the Galois group that acts on the sets $\mathcal{C}$ and $\mathrm{Irr}(G)$
    of conjugacy classes and irreducible characters of $G$. Let $\sigma \in \mathcal{G}$ be complex conjugation on $\mathrm{Irr}(G)$. It acts as inversion on $\mathcal{C}$. Let $r$ denote the rank of $R(G)$.
    Then Proposition \ref{Prop:2.1} says $r = |\mathcal{C}/ \langle \sigma \rangle| - |\mathcal{C}/\mathcal{G}|$,
    the difference of orbit numbers of $\langle \sigma \rangle$ and $\mathcal{G}$ on $\mathcal{C}$, respectively.

    By Brauer's Permutation Lemma and the Orbit Counting Lemma (see \cite[Corollary 11.10]{bible}),
    we have $|\mathcal{C}/\langle \sigma \rangle| = |\mathrm{Irr}(G)/\langle \sigma \rangle|$ and $|\mathcal{C}/\mathcal{G}| = |\mathrm{Irr}(G)/\mathcal{G}|$ and therefore $r = |\mathrm{Irr}(G)/\langle \sigma \rangle| - |\mathrm{Irr}(G)/\mathcal{G}|$.

    Now we can consider the Galois action on the set
    $\mathcal{A} = \{ c \in \mathcal{C} \ : \ c^{\langle \sigma \rangle} \not= c^{\mathcal{G}} \}$,
    and immediately get $r = |\mathcal{A}/\langle \sigma \rangle| - |\mathcal{A}/\mathcal{G}|$.
    Analogously, the $\mathcal{G}$-action on the set
    $\mathcal{B} = \{ \chi \in \mathrm{Irr}(G) \ : \ \chi^{\langle \sigma \rangle} \not= \chi^{\mathcal{G}} \}$ yields
    $r = |\mathcal{B}/\langle \sigma \rangle| - |\mathcal{B}/\mathcal{G}|$.

		Since $\mathcal{G}$ joins at least two $\langle \sigma \rangle$-orbits on $\mathcal{A}$ and $\mathcal{B}$, respectively,
    we also have $|\mathcal{A}/\langle \sigma \rangle| \geq 2 |\mathcal{A}/\mathcal{G}|$ and $|\mathcal{B}/\langle \sigma \rangle| \geq 2 |\mathcal{B}/\mathcal{G}|$.
\end{proof}

Observe that $\mathfrak{r}_{\mathbb{Z}}(G) = \sum_C \left( |C/\langle \sigma \rangle| - 1 \right)$, where the summation runs over the families $C$ of algebraic conjugate classes of $G$ (and where $|C/\mathcal{G}| = 1$ holds). This follows by applying the idea of Proposition \ref{bounded} to single families. One gets $\mathfrak{r}_{\mathbb{Z}}(G) \geq |C/\langle \sigma \rangle| - 1$ for each family. In particular, $\mathfrak{r}_{\mathbb{Z}}(G) \geq f(G)/2 - 1$.

We denote Euler's totient function by $\varphi$. For a positive integer $m = p_{1}^{k_1} \cdots p_{t}^{k_{t}}$ where $p_{1}, \ldots , p_{t}$ are distinct primes and $k_{1}, \ldots , k_{t}$ are positive integers the value of $\varphi(m)$ is $\prod_{i=1}^{t} ( p_{i}^{k_{i}} - p_{i}^{k_{i}-1} )$. It is easy to see that $\varphi(m) \geq \sqrt{m}/2$.

\begin{lemma}
\label{cyclic}
Let $H$ be a cyclic subgroup in a finite group $G$. If $n = |N_G(H)/C_G(H)|$, then
\begin{enumerate}
\item[(i)] $f(G) \geq \varphi(|H|)/n \geq \sqrt{|H|}/(2 n)$ and

\item[(ii)] $\mathfrak{r}_{\mathbb{Z}}(G) \geq (\varphi(|H|)/(2n)) - 1
\geq (\sqrt{|H|}/(4n)) - 1$.
\end{enumerate}
In particular, if $f(G) \leq 4$, then $\varphi(|H|) \leq 4 n$.
\end{lemma}

\begin{proof}
Let $h \in G$ be an element of an element of a family $C$ of algebraic conjugate classes in $G$, with $n = |N_G(H)/C_G(H)|$ where $H = \langle h \rangle$. Statement (i) follows from the paragraph preceding Lemma \ref{cyclic}. In particular, if $f(G) \leq 4$, then $\varphi(|H|) \leq 4 n$. Finally, from the two paragraphs preceding Lemma \ref{cyclic}, we get
\[
\mathfrak{r}_{\mathbb{Z}}(G) \geq \textstyle\frac{1}{2}|C| - 1 = \frac{1}{2n}\varphi(|H|) - 1 \geq \frac{1}{4n}\sqrt{|H|}- 1,
\]
giving statement (ii).
\end{proof}

\section{Computational results}
\label{computational}

In this section we collect the computational results needed to prove Theorems \ref{rank1} and \ref{mainmain}. We use the computer system \cite{GAP}. Such computations are shown already in \cite{Aleev_dissertation}.

We present a function for computing $\mathfrak{r}_{\mathbb{Z}}(G)$ for a finite group $G$. When applied to the character table of the group $G$, say, it returns the value $h_{\mathbb{R}} + (|\mathrm{Irr}(G) - h_{\mathbb{R}})/2 - n_{\mathbb{Q}}$.

\begin{lstlisting}[language=GAP]
RankOfCentralUnits:= function( tbl )
    local X, real;
    X:= Irr( tbl );
    real:= Number( X, chi -> chi = ComplexConjugate( chi ) );
    return real + ( Length( X ) - real ) / 2 - Length( RationalizedMat( X ) );
end;;
\end{lstlisting}

We use this function to determine all non-abelian finite simple groups $G$ in the character table library \cite{CTblLib} with $\mathfrak{r}_{\mathbb{Z}}(G)$ equal to $0$ or $1$.

\begin{verbatim}
gap> AllCharacterTableNames( IsSimple, true, IsAbelian, false,
>         IsDuplicateTable, false, RankOfCentralUnits, 0 );
[ "A12", "A7", "A8", "A9", "Co1", "Co2", "Co3", "HS", "L3(2)", "M", "M11",
  "M12", "M22", "M23", "M24", "McL", "O8+(2)", "S6(2)", "Th", "U3(3)",
  "U3(5)", "U4(2)", "U4(3)", "U5(2)", "U6(2)" ]
gap> AllCharacterTableNames( IsSimple, true, IsAbelian, false,
>         IsDuplicateTable, false, RankOfCentralUnits, 1 );
[ "2E6(2)", "A10", "A11", "A13", "A16", "A17", "A5", "A6", "F4(2)", "Fi22",
  "G2(3)", "L2(11)", "L3(3)", "L3(4)", "L4(3)", "O7(3)", "O8+(3)", "S6(3)",
  "S8(2)" ]
\end{verbatim}

We now present a function for computing the maximal length of Galois orbits on the set of conjugacy classes of a finite group. When applied to the character table of the group $G$, say, it returns $1$ if all conjugacy classes of $G$ are rational, and the maximum of the lengths of families of algebraic conjugates of classes of $G$ otherwise.

\begin{lstlisting}[language=GAP]
MaximalGaloisOrbitLength:= function( tbl )
local fams;
fams:= Filtered( GaloisMat( TransposedMat( Irr( tbl ) ) ).galoisfams, IsList );
return MaximumList( List( fams, l -> Length( l[1] ) ), 1 );
end;;
\end{lstlisting}

We use this function to determine all non-abelian finite simple groups $G$ in the character table library \cite{CTblLib} with maximal Galois orbit length at most $4$ on the set of conjugacy classes of $G$. The computations shown below/above take altogether only a few seconds; note that they just evaluate known character tables.

\begin{verbatim}
gap> AllCharacterTableNames( IsSimple, true, IsAbelian, false,
> IsDuplicateTable, false, MaximalGaloisOrbitLength, x -> x = 1 );
[ "O8+(2)", "S6(2)" ]
gap> AllCharacterTableNames( IsSimple, true, IsAbelian, false,
> IsDuplicateTable, false, MaximalGaloisOrbitLength, x -> x = 2 );
[ "2E6(2)", "A10", "A11", "A12", "A13", "A14", "A15", "A16", "A17", "A18",
  "A19", "A5", "A6", "A7", "A8", "A9", "B", "Co1", "Co2", "Co3", "F3+",
  "F4(2)", "Fi22", "Fi23", "G2(3)", "G2(4)", "HN", "HS", "He", "J2",
  "L2(11)", "L3(2)", "L3(4)", "M", "M11", "M12", "M22", "M23", "M24", "McL",
  "O7(3)", "O8+(3)", "S6(3)", "S8(2)", "Suz", "Th", "U3(3)", "U3(5)",
  "U4(2)", "U4(3)", "U5(2)", "U6(2)" ]
gap> AllCharacterTableNames( IsSimple, true, IsAbelian, false,
> IsDuplicateTable, false, MaximalGaloisOrbitLength, x -> x = 3 );
[ "3D4(2)", "J1", "J3", "J4", "L2(13)", "L2(17)", "L2(19)", "L2(8)", "ON",
  "Ru", "S10(2)", "S4(5)", "Sz(8)" ]
gap> AllCharacterTableNames( IsSimple, true, IsAbelian, false,
> IsDuplicateTable, false, MaximalGaloisOrbitLength, x -> x = 4 );
[ "2F4(2)'", "L2(29)", "L2(31)", "L3(3)", "L4(3)", "O10-(2)", "O8-(2)",
  "S12(2)", "S4(4)", "U3(4)" ]
\end{verbatim}

We next deal with a few groups individually.

\begin{proposition}
\label{individual}
If $G$ is any of the simple groups \;$\mathrm{PSL}(2,23)$, \;  $\mathrm{PSL}(2,27)$, \\  $\mathrm{PSL}(2,47)$, \;  $\mathrm{PSL}(2,59)$, \;  $\ly$, \;  $\mathrm{PSU}(9,2)$, \;  $\mathrm{SO}^{-}(12,2)$, \;  $\mathrm{SO}(7,5)$, \;  $\mathrm{\Omega}(9,3)$, \;  $\mathrm{\Omega}(11,3)$, \;  $\mathrm{\Omega}^{+}(8,4)$, \;  $\mathrm{P\Omega}^{+}(8,5)$, \;   $\mathrm{\Omega}^{+}(10,2)$, \;  $\mathrm{\Omega}^{+}(10,3)$, \;  $\mathrm{SO}^{+}(12,3)$, \;  $\mathrm{\Omega}^{+}(14,2)$, \;  $\mathrm{\Omega}^{+}(16,2)$, \;  then $f(G)>4$. We also have $f(\mathrm{\Omega}^{+}(12,2)) = 4$ and $\mathfrak{r}_{\mathbb{Z}}({\Omega}^{+}(12,2)) = 17$.
\end{proposition}

\begin{proof}
For $G \in \{ \mathrm{PSL}(2,23), \mathrm{PSL}(2,27), \mathrm{PSL}(2,47),
 \mathrm{PSL}(2,59), \mathrm{\Omega}^{+}(10,2), \ly \}$,
the character table is available in \cite{CTblLib},
and the \textsf{GAP} function \texttt{MaximalGaloisOrbitLength} gives
$f(G) > 4$.

For $G \in \{ \;  \mathrm{\Omega}^{+}(12,2), \;  \mathrm{SO}^{-}(12,2),  \;
\mathrm{SO}^{+}(12,3),  \;  \mathrm{SO}(7,5) \;  \}$,
the character tables have been computed with \cite{MAGMA}.
We have
\[
f(\mathrm{\Omega}^{+}(12,2)) = 4, \;
f(\mathrm{\Omega}^{-}(12,2)) = 8, \; 
f(\mathrm{SO}^{+}(12,3)) = 11, \; 
f(\mathrm{SO}(7,5)) = 6.   
\]
Moreover, the \textsf{GAP} function \texttt{RankOfCentralUnits} gives
$\mathfrak{r}_{\mathbb{Z}}({\Omega}^{+}(12,2)) = 17$.

In the remaining cases, we use the fact that two elements in a matrix group
cannot be conjugate if their characteristic polynomials differ.
Table~\ref{charpoltable} shows the relevant data.

\begin{table}
\caption{Some elements in matrix groups}
\label{charpoltable}
\[
   \begin{array}{l|l|r|r|r}
       G & \overline{G} & |g| & \#\chi_{g^i} & |Z(G)| \\ \hline
       \mathrm{SU}(9,2)        & \mathrm{PSU}(9,2)       & 255 & 16 & 3 \\
       \mathrm{\Omega}(9,3)    & G                       &  41 &  6 & 1 \\
       \mathrm{\Omega}(11,3)   & G                       & 164 &  5 & 1 \\
       \mathrm{\Omega}^+(8,4)  & G                       & 255 &  8 & 1 \\
       \mathrm{\Omega}^+(8,5)  & \mathrm{P\Omega}^+(8,5) & 312 & 12 & 2 \\
       \mathrm{\Omega}^+(10,3) & G                       & 164 &  5 & 1 \\
       \mathrm{\Omega}^+(14,2) & G                       & 127 &  9 & 1 \\
       \mathrm{\Omega}^+(16,2) & G                       & 255 &  8 & 1 \\
   \end{array}
\]
\end{table}

In all these cases, we found an element $g$ of the order in question
by taking a few (about a hundred) pseudo random elements
from the matrix group $G$,
and computed the set of characteristic polynomials of the powers $g^i$,
with $i$ coprime to $|g|$.
The number $n$, say, of these polynomials is a lower bound for the number of
those conjugacy classes in $G$ that contain generators of $\langle g \rangle$.
If $G$ is itself not simple then the factor group $\overline{G} = G/Z(G)$
is simple, and the preimage of each class of $\overline{G}$
under the natural epimorphism from $G$
consists of at most $|Z(G)|$ classes of $G$.
Thus $n/|Z(G)|$ is a lower bound for the number of conjugacy classes
in $\overline{G}$ that contain generators of $\langle g Z(G) \rangle$.

For example, there are at least $16$ classes in $G = \mathrm{SU}(9,2)$
that contain the generators of a cyclic subgroup
$\langle g \rangle$ of order $255$,
since these elements have $16$ different characteristic polynomials.
The simple group $\overline{G} = \mathrm{PSU}(9,2)$ arises from $G$
by factoring out $Z(G)$, which has order $3$.
The preimage of each class of $\mathrm{PSU}(9,2)$
under the natural epimorphism from $\mathrm{SU}(9,2)$
consists of either $1$ or $3$ classes of $\mathrm{SU}(9,2)$.
Thus $\overline{G}$ has at least $[16/3] = 5$ conjugacy classes
that contain the images of the generators of $\langle g \rangle$
under the natural epimorphism from $G$.
Hence in particular $f(\mathrm{PSU}(9,2)) > 4$.
\end{proof}

\section{Some reductions}

We begin our considerations of Theorem \ref{mainmain} and Theorem \ref{rank1}.

Let $G = C_{p}$ for some prime $p$. We have $f(G) \leq 4$ if and only if $p \in \{ 2,3,5 \}$. Moreover, $f(C_{2}) = 1$, \;  $f(C_{3}) = 2$, and $f(C_{5}) = 4$. Furthermore, $\mathfrak{r}_{\mathbb{Z}}(G) = 1$ if and only if $p = 5$ (see Proposition \ref{Prop:2.1}).

Let $G = A_{n}$ for some $n \geq 5$. It is easy to see that $f(G) = 2$. Moreover, $\mathfrak{r}_{\mathbb{Z}}(G) = 1$ if and only if $n \in \{ 5, 6, 10, 11, 13, 16, 17, 21, 25 \}$ by \cite{Aleev_Kargapolov_Sokolov}.

Theorems \ref{mainmain} and \ref{rank1} follow from the previous section and also by results of Molodorich \cite{Molodorich} in case $G$ is a sporadic simple group. Thus, in order to prove Theorem \ref{mainmain} and Theorem \ref{rank1}, we may assume that $G$ is a finite simple group of Lie type.

The computations in Section \ref{computational} are that the finite simple groups of Lie type listed in Theorem \ref{rank1} or Theorem \ref{mainmain} have the claimed properties. Thus it remains to show that no other simple group of Lie type satisfies the conditions of these theorems.



\section{Classical simple groups}
\label{classical}

In this section we prove Theorems \ref{mainmain}, \ref{rank1}, \ref{secondmain} in case $G$ is a finite simple classical group different from an alternating group. Our approach below is similar to the approach of Babai, P\'alfy, Saxl \cite[Section 4]{Babai_Palfy_Saxl}. Namely,
sufficiently large groups are eliminated by exposing elements $g$ of large order such that the normalizer
of the cyclic group $\langle g \rangle$ only induces few automorphisms.

First we prove Theorem \ref{secondmain} and one direction of Theorem \ref{mainmain} (the other direction was established in Section \ref{computational}) for finite simple classical groups.

The general linear group $\mathrm{GL}(n,q)$ contains elements of order $q^{n}-1$. These elements are called Singer elements and the cyclic subgroups generated by them are the Singer subgroups. A Singer subgroup in $\mathrm{SL}(n,q)$ is the intersection of
$\mathrm{SL}(n,q)$ with a Singer subgroup of $\mathrm{GL}(n,q)$. Singer subgroups can be defined also in other classical subgroups of $\mathrm{GL}(n,q)$ as irreducible cyclic subgroups of maximal possible orders. These are intersections of the classical groups with the Singer subgroups of $\mathrm{GL}(n,q)$. As it was described by Huppert \cite{Huppert}, there are Singer subgroups in all symplectic groups, in the  odd-dimensional unitary groups, and in the minus type orthogonal groups. The other classical groups do not contain irreducible cyclic subgroups.

According to \cite{Huppert} (see also \cite[Table 1]{Bereczky}), the group $\mathrm{SL}(n,q)$ has Singer subgroups of order $(q^{n}-1)/(q-1)$, the group $\mathrm{GU}(n/2,q)$ where $n$ is even and $n/2$ is odd has Singer subgroups of order $q^{n/2}+1$, the group $\mathrm{SU}(n/2,q)$ where $n$ is even and $n/2$ is odd has Singer subgroups of order $(q^{n/2}+1)/(q+1)$, the group $\mathrm{Sp}(n,q)$ where $n$ is even has Singer subgroups of order $q^{n/2}+1$, the groups $\mathrm{GO}^{-}(n,q)$ and $\mathrm{SO}^{-}(n,q)$ where $n$ is even have Singer subgroups of order $q^{n/2}+1$, and the group $\mathrm{\Omega}^{-}(n,q)$ where $n$ is even has Singer subgroups of order $\frac{1}{d}(q^{n/2}+1)$ where $d$ is the greatest common divisor of $2$ and $q+1$.

Let $G$ be any of the above classical groups with a Singer subgroup $C$. It is easy to see that $C_{G}(C) = C$. In particular, $Z(G) \leq C$. Moreover, $N_{G}(C)/C$ is cyclic of order dividing $n$.




We continue with the following.

\begin{proposition}
\label{c1}
Let $S$ be a non-solvable group from the following list.
\begin{enumerate}
\item[(i)] $\mathrm{SL}(n,q)$;
\item[(ii)] $\mathrm{GU}(n/2,q)$, \;  $\mathrm{SU}(n/2,q)$ with $n$ even and $n/2$ odd;
\item[(iii)] $\mathrm{Sp}(n,q)$,  \;  $\mathrm{GO}^{-}(n,q)$,  \; $\mathrm{SO}^{-}(n,q)$,  \;  $\mathrm{\Omega}^{-}(n,q)$ with $n$ even.
\end{enumerate}
There exists a universal constant $c>0$ such that both $f(S)+1$ and $\mathfrak{r}_{\mathbb{Z}}(S) + 2$ are larger than $q^{cn}$.
In particular, $f(S) \to \infty$ and $\mathfrak{r}_{\mathbb{Z}}(S) \to \infty$ as $|S| \to \infty$.
\end{proposition}

\begin{proof}
The group $S$ can naturally be embedded in $G = \mathrm{GL}(n,q)$. Let $C$ be a Singer subgroup in $G$. The group $H = S \cap C$ acts irreducibly on the underlying vector space by \cite{Huppert}. We have $C_{G}(H) = C$. On the other hand,
\[
|N_{S}(H)/C_{S}(H)| \leq |N_{G}(C)/C_{G}(C)| \leq n.
\]
Now $(q^{n/2}+1)/(q+1)\leq |H|$. It follows by Lemma \ref{cyclic} that there exists a constant $c>0$ such that both $f(S)+1$ and $\mathfrak{r}_{\mathbb{Z}}(S) + 2$ are larger than $q^{cn}$. It also follows that $f(S) \to \infty$ and $\mathfrak{r}_{\mathbb{Z}}(S) \to \infty$ as $|S| \to \infty$.
\end{proof}

We next pass from groups $S$ to $\bar{S} = S/Z$ where $Z$ denotes the center of $S$.

\begin{lemma}
\label{l1}
Let $H$ be a cyclic subgroup in a finite group $S$. Assume that the center $Z$ of $S$ is contained in $H$. Let $\bar{H}$ and $\bar{S}$ denote the groups $H/Z$ and $S/Z$. Then $N_{\bar{S}}(\bar{H})/C_{\bar{S}}(\bar{H})$ is isomorphic to a section of $N_{S}(H)/C_{S}(H)$.
\end{lemma}

\begin{proof}
If $s \in C_{S}(H)$, then $sZ \in C_{\bar{S}}(\bar{H})$ and so $C_{S}(H)/Z$ is a subgroup of $C_{\bar{S}}(\bar{H})$. Let $sZ \in N_{\bar{S}}(\bar{H})$ and let $h$ be a generator of $H$. Clearly, \[(sZ)(hZ) = (h^{m} Z)(sZ)\] for some $m\in \mathbb{Z}$. It follows that $s \in N_{S}(H)$ and so $N_{\bar{S}}(\bar{H})$ is a subgroup of $N_{S}(H)/Z$. The lemma follows.
\end{proof}



We continue with the following.

\begin{proposition}
\label{c2}
Let $\bar{S}$ be a non-abelian simple group from the following list.
\begin{enumerate}
\item[(i)] $\mathrm{PSL}(n,q)$;
\item[(ii)] $\mathrm{PSU}(n/2,q)$ with $n$ even and $n/2$ odd;
\item[(iii)] $\mathrm{PSp}(n,q)$, $\mathrm{P\Omega}^{-}(n,q)$ with $n$ even.
\end{enumerate}
There exists a universal constant $c>0$ such that $f(\bar{S}) + 1$ and $\mathfrak{r}_{\mathbb{Z}}(\bar{S}) + 2$ are larger than $q^{cn}$. In particular, $f(\bar{S}) \to \infty$ and $\mathfrak{r}_{\mathbb{Z}}(\bar{S}) \to \infty$ as $|\bar{S}| \to \infty$.
\end{proposition}

\begin{proof}
Let $S$ be a non-solvable group from the list: $\mathrm{SL}(n,q)$, $\mathrm{SU}(n/2,q)$ with $n$ even and $n/2$ odd, $\mathrm{Sp}(n,q)$ and $\mathrm{\Omega}^{-}(n,q)$ with $n$ even. Let $H$ be as in the proof of Proposition \ref{c1}.

The center $Z$ of $S$ is contained in $H$. Let $\bar{H} = H/Z$ and $\bar{S} = S/Z$. Now \[|N_{\bar{S}}(\bar{H})/C_{\bar{S}}(\bar{H})| \leq |N_{S}(H)/C_{S}(H)|\] by Lemma \ref{l1}. Thus $|N_{\bar{S}}(\bar{H})/C_{\bar{S}}(\bar{H})| \leq n$ again by the proof of Proposition \ref{c1}.

If $\bar{S} \not= \mathrm{PSU}(n/2,q)$, then $\frac{1}{4n}(q^{n/2}+1) \leq |H/Z|$.

If $\bar{S} = \mathrm{PSU}(n/2,q)$, then $(q^{n/2}+1)/(d(q+1)) \leq |H/Z|$ where $d$ is the greatest common divisor of $n/2$ and $q+1$.

These two lower bounds and Lemma \ref{cyclic} imply that there exists a constant $c>0$ such that $f(\bar{S})+1$ and $\mathfrak{r}_{\mathbb{Z}}(\bar{S}) + 2$ are larger than $q^{cn}$. It also follows that $f(\bar{S}) \to \infty$ and $\mathfrak{r}_{\mathbb{Z}}(\bar{S}) \to \infty$ as $|\bar{S}| \to \infty$.
\end{proof}

Let $\bar{S}$ be any of the groups listed in the statement of Proposition \ref{c2}.

If $\bar{S} = \mathrm{PSL}(n,q)$ is (non-abelian) simple and different from an alternating group, then
$4n < \varphi\big((q^{n}-1)/(d(q-1)) \big)$ where $d = (n,q-1)$ unless $(n,q)$ is in
\[
\{ \;  (2,7), (2,8), (2,11), (2,13), (2,17), (2,19), (2,23), (2,27), (2,29), (2,31),$$ $$(2,47), (2,59), (3,2), (3,3), (3,4), (4,3) \;  \}.
\]
All projective special linear groups with such exceptional parameters appear in the list of Theorem \ref{mainmain} (and thus all Galois orbits on the set of conjugacy classes have size at most $4$ by Section \ref{computational}), apart from $\bar{S} = \mathrm{PSL}(2,23)$, $\mathrm{PSL}(2,27)$, $\mathrm{PSL}(2,47)$, $\mathrm{PSL}(2,59)$ in which cases $f(\bar{S}) > 4$ by Proposition \ref{individual}.

Let $\bar{S} = \mathrm{PSp}(n,q)$ be simple with $n \geq 4$ even. This group contains a cyclic subgroup of order $\frac{1}{d}(q^{n/2}+1)$ where $d = (2,q-1)$. We have $\varphi(\frac{1}{d}(q^{n/2}+1)) > 4n$ unless $$(n,q) \in \{ (4,3), (4,4), (4,5), (6,2), (6,3), (8,2), (10,2), (12,2) \}.$$ All arising exceptional groups appear in the list of Theorem \ref{mainmain} and all Galois orbits on the set of conjugacy classes have size at most $4$ by Section \ref{computational}.

Let $\bar{S} = \mathrm{PSU}(n/2,q)$ be simple with $n/2 \geq 3$ odd. This group contains a cyclic subgroup of order $(q^{n/2}+1)/(d(q+1))$ where $d =(n/2, q+1)$. We have $\varphi((q^{n/2}+1)/(d(q+1))) > 2n$ unless
\[
(n,q) \in \{ \;  (6,3), (6,4), (6,5), (10,2), (18,2)  \; \}.
\]
The first four groups arising appear in the list of Theorem \ref{mainmain} and the corresponding values for $f$ are at most $4$ by Section \ref{computational}. The group $\mathrm{PSU}(9,2)$ satisfies $f(\mathrm{PSU}(9,2)) > 4$ by Proposition \ref{individual}.

Let $\bar{S} = \mathrm{P\Omega^{-}}(n,q)$ be simple with $n \geq 8$ even. There is a cyclic subgroup of order $\frac{1}{d}(q^{n/2}+1)$ where $d = (2,q+1)$ in $\bar{S}$. We have $\varphi(\frac{1}{d}(q^{n/2}+1)) > 4n$ unless $(n,q) \in \{ (8,2), (10,2), (12,2) \}$. The first two groups
arising appear in the list of Theorem \ref{mainmain} and all Galois orbits on the set of conjugacy classes have size at most $4$ by Section \ref{computational}. The group $\mathrm{SO^{-}}(12,2)$ satisfies $f(\mathrm{SO^{-}}(12,2)) > 4$ by Proposition \ref{individual}.

In order to deal with the remaining simple classical groups, we need a lemma.

\begin{lemma}
\label{directsum}
Let $V$ be a vector space of dimension $n \geq 4$ defined over a finite field $F$. Let $g$ be an element in $\mathrm{GL}(V)$ such that the $F \langle g \rangle$-module $V$ may be decomposed as $V = U \oplus W$ where $U$ and $W$ are irreducible $F \langle g \rangle$-modules with the property that
\[
(\dim(U), \dim(W)) \in \{\, (n-1,1),\; (n-2,2)\, \}.
\]
Then $|N_{S}(\langle g \rangle)/C_{S}(\langle g \rangle)| \leq 2n$ for any subgroup $S$ of $\mathrm{GL}(V)$ with $g \in S$.
\end{lemma}

\begin{proof}
Let $C = \langle g \rangle$ and $G = \mathrm{GL}(V)$. We may assume that $S = G$. Let $h \in N_{G}(C)$. Since
\[
Uh = (UC)h = U(Ch) = U(hC) = (Uh)C,\] the vector space $Uh$ is an $FC$-submodule of $V$. Similarly, $Wh$ is also an $FC$-submodule of $V$. Since $|C|$ is coprime to $|F|$, the $FC$-module $V$ is completely reducible and so \[\{ Uh,\; Wh \} = \{ U,\; W \}.\]
It follows that the group $N_{G}(C)$ has a subgroup $N$ of index at most $2$ such that both $U$ and $W$ are $FN$-submodules of $V$. If
\[
(\dim(U), \dim(W)) \in \{\, (n-1,1),\;  (n-2,2)\, \}
\]
with $n > 4$, then $N_{G}(C) = N$ and $|N_{G}(C)/C_{G}(C)| \leq 2(n-2)$.

If $\dim(U) = \dim(W) = 2$, then $|N/C_{G}(C)| \leq 4$ and so $|N_{G}(C)/C_{G}(C)| \leq 8$.
\end{proof}

\begin{proposition}
\label{c3}
Let $\bar{S}$ be a non-abelian simple group from the following list.
\begin{enumerate}
\item[(i)] $\mathrm{PSU}(n/2,q)$ with $n$ divisible by $4$;
\item[(ii)] $\mathrm{P\Omega}(n,q)$ with $n$ odd;
\item[(iii)] $\mathrm{P\Omega}^{+}(n,q)$ with $n$ even.
\end{enumerate}
There exists a universal constant $c>0$ such that $f(\bar{S})+1$ and $\mathfrak{r}_{\mathbb{Z}}(\bar{S}) + 2$ are larger than $q^{cn}$. In particular, $f(\bar{S}) \to \infty$ and $\mathfrak{r}_{\mathbb{Z}}(\bar{S}) \to \infty$ as $|\bar{S}| \to \infty$.
\end{proposition}

\begin{proof}
Let $S$ be the group $\mathrm{SU}(n/2,q) \leq \mathrm{GL}(V)$ where $n$ is divisible by $4$ and $V$ is the vector space of dimension $n/2$ defined over the field of size $q^{2}$ equipped with a non-singular conjugate-symmetric sesquilinear form. There is a non-singular subspace $U$ of $V$ of dimension $(n/2)-1$. The vector space $W = U^{\perp}$ is a non-singular subspace of dimension $1$ and $V = U \oplus W$ as vector spaces. Since $\dim(U)$ and $\dim(W)$ are both odd, there is an element $g \in S$ of maximal possible order by \cite{Huppert} such that $H = \langle g \rangle$ acts irreducibly on both $U$ and $W$. Moreover we may choose $g$ in such a way that $H$ contains the center $Z$ of $S$. Let $\bar{S} = S/Z$ and $\bar{H} = H/Z$. It follows that
\[
|N_{\bar{S}}(\bar{H})/C_{\bar{S}}(\bar{H})| \leq |N_{S}(H)/C_{S}(H)| \leq n
\]
by Lemmas \ref{l1} and \ref{directsum}. Now $\frac{1}{d}\big (q^{(n/2)-1}+1\big ) = |\bar{H}|$ where $d$ denotes the greatest common divisor of the numbers $n/2$ and $q+1$. There exists a constant $c_{1} > 0$ such that $q^{c_{1} n} < \frac{1}{2n} \varphi(|\bar{H}|) - 1$, by the paragraph preceding Lemma \ref{cyclic}, provided that the right-hand side is larger than $1$. Thus $f(\bar{S}) +1$ and $\mathfrak{r}_{\mathbb{Z}}(\bar{S}) + 2$ are larger than $q^{c_{1}n}$ by Lemma \ref{cyclic}. In particular, $f(\bar{S}) \to \infty$ and $\mathfrak{r}_{\mathbb{Z}}(\bar{S}) \to \infty$ as $|\bar{S}| \to \infty$.

Now we view $V$ as a vector space of dimension $n$ defined over the field of size $q$. Let $V$ be equipped with a non-singular quadratic form. Let $S$ be $\mathrm{\Omega}(n,q) \leq \mathrm{GL}(V)$ with $n$ and $q$ odd, or let $S$ be $\mathrm{\Omega}^{+}(n,q) \leq \mathrm{GL}(V)$ with $n$ even. As described in \cite[p. 75]{Wilson}, there are subgroups $\mathrm{\Omega}^{-}(n-1,q) \times \mathrm{\Omega}(1,q)$ and $\mathrm{\Omega}^{-}(n-2,q) \times \mathrm{\Omega}^{-}(2,q)$ in $S$, in the respective cases, preserving a decomposition $V = U \oplus W$ where $U$ and $W$ are non-singular subspaces of $V$ with
\[
(\dim(U), \dim(W))\in\{\,(n-1,1),\; (n-2,2)\,\}.
\]
Let $S$ be any of the orthogonal groups considered in this proof. Let $\bar{S} = S/Z$ where $Z$ is the center of $S$. As in the previous paragraph, there is a cyclic group $H$ with $Z \leq H \leq S$ such that $H$ acts irreducibly on both subspaces $U$ and $W$ and
\[
|N_{\bar{S}}(\bar{H})/C_{\bar{S}}(\bar{H})| \leq 2n
\]
where $\bar{H} = H/Z$. We have $\frac{1}{4}\big(q^{(n-2)/2}+1\big)\leq |\bar{H}|$. There exists a constant $c_{2} > 0$ such that $q^{c_{2} n} < \frac{1}{4n} \varphi(|\bar{H}|) - 1$, by the paragraph preceding Lemma \ref{cyclic}, provided that the right-hand side is larger than $1$. Thus $f(\bar{S})+1$ and $\mathfrak{r}_{\mathbb{Z}}(\bar{S}) + 2$ are larger than $q^{c_{2}n}$ by Lemma \ref{cyclic}. In particular, $f(\bar{S}) \to \infty$ and $\mathfrak{r}_{\mathbb{Z}}(\bar{S}) \to \infty$ as $|\bar{S}| \to \infty$.

Finally, let $c$ be the minimum of $c_1$ and $c_2$.
\end{proof}

Let $\bar{S}$ be as in the statement of Proposition \ref{c3}.

Let $\bar{S} = \mathrm{PSU}(n/2,q)$ with $n \geq 8$ divisible by $4$. This group contains a cyclic subgroup of order divisible by $\frac{1}{d}(q^{(n/2)-1}+1)$ where $d = (n/2, q+1)$. We have
\[
 \varphi\big(\textstyle\frac{1}{d}(q^{(n/2)-1}+1)\big)> 4n
\]
unless $(n,q) \in \{ (8,2), (8,3), (12,2) \}$. Since $\mathrm{SU}(4,2) \cong \mathrm{PSp}(4,3)$, all three arising exceptional groups $\bar{S}$ satisfy $f(\bar{S}) \leq 4$ by Section \ref{computational} and appear in the list of Theorem \ref{mainmain}.

Let $\bar{S} = \mathrm{P\Omega}(n,q)$ be simple with $n \geq 7$ and $q$ both odd. This group contains a cyclic subgroup of order divisible by $\frac{1}{2} (q^{(n-1)/2}+1)$. We have
\[
 \varphi\big(\textstyle\frac{1}{2} (q^{(n-1)/2}+1) \big)> 8n.
\]
unless $(n,q) \in \{ (7,3), (7,5), (9,3), (11,3) \}$. We have $f(\mathrm{P\Omega}(7,3)) \leq 4$ by Section \ref{computational}. The remaining three groups $\bar{S}$ satisfy $f(\bar{S}) > 4$ by Proposition \ref{individual}.

Finally, let $\bar{S} = \mathrm{P\Omega}^{+}(n,q)$ with $n \geq 8$ even. This group contains a cyclic subgroup of order divisible by $\frac{1}{d} (q^{(n-2)/2}+1)$. We have
$$\varphi\big(\textstyle\frac{1}{d} (q^{(n-2)/2}+1)\big)> 8(n-2)$$
by the proof of Lemma \ref{directsum}, where $d = (2,q+1)$ unless
\[
(n,q) \in \{ (8,2), (8,3), (8,4), (8,5), (10,2), (10,3), (12,2), (12,3), (14,2), (16,2) \}.
\] The first two groups $\bar{S}$ arising, together with $\bar{S} = \mathrm{\Omega}^{+}(12,2)$, satisfy $f(\bar{S}) \leq 4$ by Section \ref{computational} and appear in the list of Theorem \ref{mainmain}. All other arising groups $\bar{S}$, apart from $\mathrm{\Omega}^{+}(12,2)$, satisfy $f(\bar{S}) > 4$ by Proposition \ref{individual}.

This finishes the proof of Theorems \ref{secondmain} and \ref{mainmain} in case $G$ is a classical group.

It remains to show Theorem \ref{rank1} for classical groups. It is sufficient to prove, by Section \ref{computational}, that if $G$ is a finite simple classical group not in the list of the statement of Theorem \ref{rank1}, then $\mathfrak{r}_{\mathbb{Z}}(G) > 1$. Moreover, since Theorem \ref{mainmain} is established for classical groups, it is sufficient to show that if $G$ is a classical group appearing in the list of Theorem \ref{mainmain} but not appearing in the list of the statement of Theorem \ref{rank1}, then $\mathfrak{r}_{\mathbb{Z}}(G) > 1$. This follows from the content of Section \ref{computational}.

\section{Simple groups of Lie type}
\label{exceptional}

An aim of this section is to complete the proofs of Theorems \ref{mainmain} and \ref{rank1} by dealing with the remaining class of finite simple groups, the exceptional simple groups of Lie type.

The next proposition is a consequence of Tables I and II in the paper \cite{Babai_Palfy_Saxl} of Babai, P\'alfy and Saxl.

\begin{proposition}
\label{propo}
Let $\bar{S}$ be an exceptional simple group of Lie type defined over the field of size $q$ or the Tits group $\tF{2}'$. Then $f(\bar{S}) \leq 4$ if and only if
\[
\bar{S} \in \{\;  \tD{2},\; \tE{2},\;  \mathrm{F}_{4}(2),\;  \mathrm{G}_{2}(3),\;  \mathrm{G}_{2}(4),
 \; \tB{8}, \; \tF{2}' \; \}.
\]
 It follows that $\mathfrak{r}_{\mathbb{Z}}(\bar{S}) \geq 2$ unless $\bar{S} \in \{ \tE{2}, \mathrm{F}_{4}(2), \mathrm{G}_{2}(3) \}$ when $\mathfrak{r}_{\mathbb{Z}}(\bar{S}) = 1$. Moreover, $f(\bar{S})+1$ and $\mathfrak{r}_{\mathbb{Z}}(\bar{S}) + 2$ are larger than $q^{c}$ for some universal constant $c>0$. In particular, if $|\bar{S}| \to \infty$, then $f(\bar{S}) \to \infty$ and $\mathfrak{r}_{\mathbb{Z}}(\bar{S}) \to \infty$.
\end{proposition}

\begin{proof}
Tables I and II in \cite{Babai_Palfy_Saxl} give information about cyclic tori $T$ in exceptional groups $\bar{S}$ of Lie type including the Tits group $\bar{S} = \tF{2}'$. The groups $T$ satisfy $C_{\bar{S}}(T) = T$ and the middle columns in the tables provide $|T|$. The $|T|$ are polynomials in $q$ where $q$ is the size of the field of definition for $\bar{S}$. The exact values of $|N_{\bar{S}}(T)/T|$ are also provided in \cite[Tables I and II]{Babai_Palfy_Saxl}. All such entries are at most $30$. The third and fourth statements of the proposition follow from Lemma \ref{cyclic}.

Also, \cite[Tables I and II]{Babai_Palfy_Saxl} shows that for every $\bar{S}$ there is a $T$ such that $4 |N_{\bar{S}}(T)/T|< \varphi(|T|)$ unless
\[
\bar{S} \in \{ \; \tD{2}, \; \tE{2}, \; \mathrm{F}_{4}(2), \; \mathrm{G}_{2}(3), \;
\mathrm{G}_{2}(4), \; \tB{8}, \; \tF{2}' \; \}.
\]
These seven groups $\bar{S}$ satisfy $f(\bar{S}) \leq 4$ by Section \ref{computational}. The first statement now follows from Lemma \ref{cyclic}. Moreover, among these seven groups only $\tE{2}$, $\mathrm{F}_{4}(2)$ and $\mathrm{G}_{2}(3)$ satisfy $\mathfrak{r}_{\mathbb{Z}}(\bar{S}) = 1$ and $\mathfrak{r}_{\mathbb{Z}}(\bar{S}) \geq 2$ for the others, by Section \ref{computational}. The second statement follows.
\end{proof}

The other aim of this section is to summarize some of our results on simple groups of Lie type.

\begin{proposition}
\label{lie}
There exists a universal constant $c > 0$ such that whenever $G$ is a finite simple group of Lie type of Lie rank $r$ defined over the field of size $q$ then $\mathfrak{r}_{\mathbb{Z}}(G) > q^{cr}$ provided that $\mathfrak{r}_{\mathbb{Z}}(G) > 1$.
\end{proposition}

\begin{proof}
This follows from Propositions \ref{c2}, \ref{c3} and \ref{propo}.
\end{proof}

\section{Alternating groups}

As a continuation of \cite{Jucys}, Patay was first to consider the structure of central units of $\mathbb{Z}A_n$ in the case when $A_n$ is the alternating group of degree $n$.

By results of Frobenius \cite{Frob1} and formula (\ref{E:1}), the number $\mathfrak{r}_{\mathbb{Z}}(A_{n})$ is equal to the number of partitions $\lambda = (\lambda_{1}, \ldots , \lambda_{k})$ of $n$ that satisfy the following conditions: (a) $\lambda_{i}$ is odd, $1 \leq i \leq k$; (b) the $\lambda_{i}$ are pairwise distinct; (c) $n \equiv k \pmod 4$; (d) $\prod_{i=1}^{k} \lambda_{i}$ is not a square. This fact was well-known in the 1970's and it was used by Patay to classify groups $A_n$ with $\mathfrak{r}_{\mathbb{Z}}(A_{n}) = 0$.

There is an ``experimental formula'' in \cite[p. 166]{Aleev_Kargapolov_Sokolov} for the behavior of $\mathfrak{r}_{\mathbb{Z}}(A_{n})$ for large $n$. Here we prove a weaker statement.

\begin{proposition}
\label{alternating}
If $n \geq 26$, then $\mathfrak{r}_{\mathbb{Z}}(A_{n}) > c^{\sqrt{n}}$ for some universal constant $c > 1$.
\end{proposition}

\begin{proof}
Let $n$ be an integer at least $26$. Then $\mathfrak{r}_{\mathbb{Z}}(A_{n}) > 1$ by \cite{Aleev_Kargapolov_Sokolov}. In order to prove the claim, we may assume that $n$ is sufficiently large.

Let $p$ be the smallest prime greater than $n/2$. By the prime number theorem, we may assume that $p < 2n/3$ provided that $n$ is sufficiently large. Let $k$ be an integer $x$ such that $|x - (\sqrt{p}/10)|$ is as small as possible and $x$ is congruent to $n$ modulo $4$. Let $m$ be the integer satisfying
\[
2m + k^{2} -1 = n-p.
\]
By our condition $p < 2n/3$ and the definition of $k$, we have $m > n/10$ provided that $n$ is sufficiently large.

The number of $(k-1)$-tuples $(x_{1}, \ldots , x_{k-1})$ of positive integers $x_{1}, \ldots , x_{k-1}$ such that $m = x_{1} + \cdots + x_{k-1}$ is $\binom{m-1}{k-2}$. There are at most $(k-1)!$ ways to order the $x_i$. Thus the number of partitions of $m$ into exactly $k-1$ parts is at least
\[
\textstyle\frac{1}{(k-1)!} \binom{m-1}{k-2} \geq {\Big(\frac{1}{k-1}\Big)}^{k-2} {\Big(\frac{m-1}{k-2}\Big)}^{k-2}.
\]
This is at least $c^{\sqrt{n}}$ for some constant $c > 1$, provided that $n$ is sufficiently large.

We claim that $\mathfrak{r}_{\mathbb{Z}}(A_{n})$ is at least the number of partitions of $m$ into exactly $k-1$ parts. For this we use the description of certain partitions found before the statement of this proposition.

Let $\pi$ be a partition of $m$ into exactly $k-1$ parts. For each $i$ with $1 \leq i \leq k-1$, add $i$ to the $i$-th smallest part in $\pi$. Let the resulting partition be $\pi'$. Now multiply each part of $\pi'$ by $2$ and add $1$. Let the resulting partition be $\pi''$. Finally, add a part equal to $p$ to $\pi''$ to obtain the partition $\pi'''$. This is a partition of $p+2m+k^{2}-1 = n$ into exactly $k$ parts each of distinct odd lengths such that the product of the parts is not a square (since $p$ divides this number but $p^{2}$ does not). Finally, if $\pi_{1}$ and $\pi_{2}$ are partitions of $m$ into exactly $k-1$ parts providing the partition $\pi'''$ in the described way, then $\pi_{1} = \pi_{2}$.
\end{proof}

\section{Proof of Theorem \ref{secondmain}}

Let $k(G)$ denote the number of conjugacy classes of the finite group $G$. This is equal to the number of complex irreducible characters of $G$. For any finite group $G$ we have $\mathfrak{r}_{\mathbb{Z}}(G) < k(G)$ by (\ref{E:1}). The theorem is true for $G$ a cyclic group of prime order, again by (\ref{E:1}).

In order to prove our statement, we may assume that $G$ is a sufficiently large non-abelian finite simple group. In particular, we may assume that $G$ is not a sporadic simple group. Thus $G$ is an alternating group or a simple group of Lie type. The last statement follows from Propositions \ref{alternating} and \ref{lie} (together with \cite[Theorem 5.1]{BCJM} and Theorem \ref{rank1}).

We have $k(A_{n}) < d^{\sqrt{n}}$ for some universal constant $d>1$ by classical results on the number of partitions of $n$. If $G$ is a finite simple group of Lie type of Lie rank $r$ defined over the field of size $q$, then $k(G) \leq  {(6 q)}^{r} < q^{4r}$ by \cite[Theorem 1]{LiebeckPyber}. (For an improvement of this latter bound see \cite{FulmanGuralnick}.) From these and Propositions \ref{alternating} and \ref{lie} it follows that there exists a universal constant $c'>0$ such that ${k(G)}^{c'} < \mathfrak{r}_{\mathbb{Z}}(G)$ for any finite simple group $G$ with $\mathfrak{r}_{\mathbb{Z}}(G) > 1$.

\section*{Acknowledgement}
We thank Rifkhat Z. Aleev for some discussions on this topic especially during the preparation of the first version of this paper.

\end{document}